\documentclass{article}
\usepackage{amsmath}
\usepackage{latexsym}
\usepackage{amssymb}
\usepackage{colonequals}
\newtheorem{thm}{Theorem}[section]
\newtheorem{la}[thm]{Lemma}
\newtheorem{Defn}[thm]{Definition}
\newtheorem{Remark}[thm]{Remark}
\newtheorem{Conj}[thm]{Conjecture}
\newtheorem{prop}[thm]{Proposition}

\newtheorem{Example}[thm]{Example}
\newtheorem{Number}[thm]{\!\!}
\newenvironment{defn}{\begin{Defn}\rm}{\end{Defn}}

\newenvironment{rem}{\begin{Remark}\rm}{\end{Remark}}

\newenvironment{numba}{\begin{Number}\rm}{\end{Number}}
\newenvironment{proof}{{\noindent\bf Proof.}}%
                  {\nopagebreak\hspace*{\fill}$\Box$\medskip\par}

\newcommand{\cA}{{\mathcal A}}
\newcommand{\cF}{{\mathcal F}}

\newcommand{\ve}{\varepsilon}
\newcommand{\R}{{\mathbb R}}
\newcommand{\N}{{\mathbb N}}

\newcommand{\mto}{\mapsto}
\newcommand{\sub}{\subseteq}

\newcommand{\wb}{\overline}
\DeclareMathOperator{\id}{id}
\DeclareMathOperator{\Spann}{span}
\DeclareMathOperator{\GL}{GL}
\DeclareMathOperator{\Sym}{Sym}

\DeclareMathOperator{\pr}{pr}

\DeclareMathOperator{\Diff}{Diff}
\DeclareMathOperator{\Lip}{Lip}
\DeclareMathOperator{\op}{op}
\DeclareMathOperator{\str}{str}
\DeclareMathOperator{\face}{fr}
\DeclareMathOperator{\algint}{algint}
\DeclareMathOperator{\aff}{aff}
\DeclareMathOperator{\ind}{ind}
\begin{document}
\begin{center}
{\large\bf Diffeomorphism groups of convex polytopes}\\[3mm]
{\bf Helge Gl\"{o}ckner}\vspace{1mm}
\end{center}
\begin{abstract}
\noindent
Let $M$ be a convex polytope
in $\R^n$, with non-empty interior.
We turn the group $\Diff(M)$
of all $C^\infty$-diffeomorphisms
of~$M$ into a regular
Lie group.
\end{abstract}
\noindent
{\small{\bf Keywords and phrases.}
Polytope, polyhedron, diffeomorphism group,
infinite-dimensional Lie group, regularity,
manifold with corners, local addition\\[2mm]
% polyhedral manifolds
(secondary)}
\section{Introduction and statement of results}
Let $E$ be a finite-dimensional real vector space
and $M\sub E$
be the convex hull of a finite subset of~$E$,
with non-empty interior~$M^0$.
For each $x\in M$,
there exists a smallest face $M(x)$ of~$M$
such that $x\in M(x)$;
we set
\[
E(x):=\Spann_\R(M(x)-x)\sub E .
\]
For $k\in \N_0\cup\{\infty\}$,
endow the space $C^k(M,E)$
of all $C^k$-maps
$f\colon M\to E$ with the compact-open
$C^k$-topology (as in \cite[Section 1.7]{GaN});
then
\[
C^k_{\str}(M,E):=\{f\in C^k(M,E)\colon(\forall x\in M)\;
f(x)\in E(x)\}
\]
is a closed vector subspace
(on which the induced
topology will be used).
We regard the functions $f\in C^k(M,E)$
as $C^k$-vector fields on~$M$;
the elements $f\in C^k_{\str}(M,E)$
will be called \emph{stratified} $C^k$-vector fields
(cf.\ \cite{Eyn}).
Let
\[
\Diff^k(M):=\{\phi\in C^k(M,M)\colon (\exists \psi\in C^k(M,M))\colon
\phi\circ\psi=\psi\circ\phi=\id_M\}
\]
be the set of all
$C^k$-diffeomorphisms
$\phi\colon M\to M$,
and abbreviate $\Diff(M):=\Diff^\infty(M)$.
Then $\Diff^k(M)$ is a group,
using composition
of diffeomorphisms as the group multiplication.
The neutral element is the identity map~$\id_M$.
Our main goal is the following result:
\begin{thm}\label{main-thm}
The group $\Diff(M)$ admits a smooth manifold structure
modeled on $C^\infty_{\str}(M,E)$
making it a Lie group.
\end{thm}
More details concerning the Lie group
structure on $\Diff(M)$ are now described.
If $P\sub E$ is a convex polytope,
we call a homeomorphism
$\phi\colon P\to P$
\emph{face respecting}
if $\phi(F)=F$ for each face~$F$ of~$P$.
Then also $\phi^{-1}$ is
face respecting.
For each $k\in \N\cup\{\infty\}$,
the set
\[
\Diff^k_{\face}(M):=\{\phi\in \Diff^k(M)\colon
\mbox{$\phi$ is face respecting}\}
\]
is a normal subgroup of finite index in
$\Diff^k(M)$ (see Lemma~\ref{fin-index})
and
\[
\Omega_k:=\Diff^k_{\face}(M)-\id_M
\]
is an open $0$-neighbourhood in $C^k_{\str}(M,E)$
(see Lemma~\ref{omega-open}).
We give $\Diff^k_{\face}(M)$ the smooth manifold structure
modeled on $C^k_{\str}(M,E)$
making the bijective map
\[
\Phi_k\colon
\Diff^k_{\face}(M)\to\Omega_k,
\quad\phi\mto\phi-\id_M
\]
a $C^\infty$-diffeomorphism.
Then the following holds,
using $C^{k,\ell}$-maps as in~\cite{AaS}:
\begin{prop}\label{main-prop}
For all $k\in\N\cup\{\infty\}$ and $\ell\in \N_0\cup\{\infty\}$,
the map
\[
c_{k,\ell}\colon
\Diff^{k+\ell}_{\face}(M)\times \Diff^k_{\face}(M)\to\Diff^k_{\face}(M),\quad
(\phi,\psi)\mto\phi\circ\psi
\]
is $C^{\infty,\ell}$ $($and thus $C^\ell)$; moreover, the map
\[
\iota_{k,\ell}\colon \Diff^{k+\ell}_{\face}(M)\to\Diff^k_{\face}(M),\quad
\phi\mto\phi^{-1}
\]
is $C^\ell$.
Notably, $\Diff^\infty_{\face}(M)$ is a smooth Lie group
modeled on $C^\infty_{\str}(M,E)$.
\end{prop}
To establish Theorem~\ref{main-thm}, we shall give
$\Diff(M)$
a smooth Lie group structure modeled
on $C^\infty_{\str}(M,E)$
which turns $\Diff^\infty_{\face}(M)$
into an open subgroup.
\begin{rem}
The differentiability properties of the mappings $c_{k,\ell}$
and $\iota_{k,\ell}$ established in Proposition~\ref{main-prop}
show that $\Diff^\infty_{\face}(M)$
fits into the framework of \cite{BaH}.
Thus $\Diff^\infty_{\face}(M)$ (and hence
also $\Diff(M)$)
is an $L^1$-regular Lie group
in the sense of~\cite{MRE}, by~\cite{GSU}. Notably, the Fr\'{e}chet-Lie
group $\Diff(M)$
is $C^0$-regular (as in~\cite{SER})
and hence regular
in the sense of~\cite{Mil}.
\end{rem}
\begin{rem}
The easiest -- but most relevant -- case of our construction
is the case of a cube $M:=[0,1]^n\sub\R^n$.
The corresponding diffeomorphism group is of interest
in numerical mathematics~\cite{CGS}.
We mention that $[0,1]^n$ is a prime
example of a smooth manifold with corners
(as in \cite{Cer,Dou,Mic}; cf.\ \cite{MRO}
for infinite-dimensional generalizations).
Michor~\cite{Mic} discussed the diffeomorphism
group of a paracompact, smooth manifold~$M$
with corners, but his arguments contain a serious flaw:
Contrary to claims in the book,
local additions in Michor's sense
never exist if $\partial M\not=\emptyset$
(not even for $M=[0,\infty[$
or $M=[0,1]$), as we show in an appendix.
Yet, Michor's conclusion is correct
that $\Diff(M)$ is a Lie group.
More generally, $\Diff(M)$
can be made a Lie group for each
paracompact, locally polyhedral $C^\infty$-manifold~$M$
(as in Remark~\ref{polyhedral-mfd});
details will be given elsewhere.
\end{rem}
As before, let $E$ be a finite-dimensional real
vector space.
If $M\sub E$
is any compact convex subset with non-empty interior,
then a Lie group structure
can be constructed on the group $\Diff_{\partial M}(M)$
of all $C^\infty$-diffeomorphisms $\phi\colon M\to M$
such that $\phi(x)=x$ for all $x\in\partial M$
(see \cite{GN2}). In the case that~$M$
is a polytope, our discussion of the Lie group
structure on $\Diff_{\face}^\infty(M)$
proceeds along similar lines.
Then $\Diff_{\partial M}(M)$
is a closed normal subgroup and smooth submanifold
of both $\Diff_{\face}^\infty(M)$
and $\Diff(M)$ (see Remark~\ref{boundary-fix}).
% also SPLIT Lie subgroup?
For numerical applications, it
is useful to know a lower bound
for the $0$-neighbourhoods $\Omega_k\sub C^k_{\str}(M,E)$.
Fixing any norm $\|\cdot\|$ on $E$
to calculate operator norms,
we show:
\begin{prop}\label{prop-2}
Let $M\sub E$ be a convex polytope
with non-empty interior
and $k\in\N\cup\{\infty\}$.
Let $U_k$ be the set of all
$f\in C^k_{\str}(M,E)$
such that
\[
\|f\|_{\infty,\op}:=\sup_{x\in M}\|f'(x)\|_{\op}<1.
\]
Then $U_k$ is an open $0$-neighbourhood in $C^k_{\str}(M,E)$
and $U_k\sub\Omega_k$.
\end{prop}
For $E=\R^n$ and $M=[0,1]^n$, see
already~\cite{CGS}.
An analogous result for $\Diff_{\partial M}(M)$
was also established in~\cite{CGS},
for compact convex
subsets $M \sub \R^n$
with non-empty~interior.
\begin{rem}
Note that $\|f\|_{\infty,\op}$
is the smallest Lipschitz constant $\Lip(f)$
for~$f$
in the situation of Proposition~\ref{prop-2}
(see, e.g., \cite[Lemma~1.5.3\,(c)]{GaN}).
\end{rem}
\section{Preliminaries and basic facts}
We write $\N=\{1,2,\ldots\}$
and $\N_0:=\N\cup\{0\}$.
The term ``locally convex space''
means a locally convex, Hausdorff
topological vector space over the ground field~$\R$.
If $(E,\|\cdot\|)$
is a normed space, let
$B^E_r(x):=\{y\in E\colon \|y-x\|<r\}$
and $\wb{B}^E_r(x):=\{y\in E\colon \|y-x\|\leq r\}$
for $x\in E$ and $r>0$.
For background concerning
convex polytopes
and their faces,
the reader is referred to \cite{Bro}
and \cite{Web}.
If $E$ is a real vector space and
$S\sub E$ a subset,
we write $\aff(S)$ for the affine
subspace of~$E$ generated by~$S$.
In this article,
we are working in the setting of infinite-dimensional
calculus known as Keller's $C^k_c$-theory
(going back to \cite{Bas}).
Accordingly, the smooth manifolds and Lie groups
we consider are modeled on locally convex
spaces which need not have finite dimension;
see \cite{RES}, \cite{GaN},
and \cite{NSU}
for further information.\\[2.3mm]
Let $E$ and $F$ be locally convex spaces.
We recall from~\cite[Chapter~1]{GaN}:
\begin{numba}
A map $f\colon U\to E$
on an open subset $U\sub E$
is called $C^1$ if $f$ is continuous,
the directional derivative
\[
df(x,y):=(D_yf)(x):=\lim_{t\to 0}\frac{1}{t}\left( f(x+ty)-f(x)\right)
\]
exists in~$F$ for all $(x,y)\in U\times E$,
and the map $df\colon U\times E\to F$
is continuous.
\end{numba}
Then $f'(x):=df(x,\cdot)\colon E\to F$ is linear for each $x\in U$.
\begin{numba}
Let $U\sub E$ be a subset which is locally convex
(in the sense that each $x\in U$
has a convex neighbourhood in~$U$)
and whose interior
$U^0$ is dense in~$U$.
A map $f\colon U\to F$
is called $C^0$ if $f$ is continuous.
If $f$ is continuous,
$f|_{U^0}$ is $C^1$ and $d(f|_{U^0})\colon U^0\times E\to F$
has a continuous extension
\[
df\colon U\times E\to F,
\]
then $f$ is called~$C^1$.
Given $k\in\N$, we say that
$f$ is $C^k$ if $f$ is $C^1$ and $df$ is $C^{k-1}$.
If $f$ is $C^k$ for all $k\in\N_0$,
then $f$ is called a $C^\infty$-map
or also \emph{smooth}.
\end{numba}
If $E$ and $F$ are locally convex spaces,
we write $L(E,F)_b$
for the locally convex space of all continuous
linear maps $\alpha\colon E\to F$,
endowed with the topology of uniform
convergence on bounded subsets of~$E$.
For finite-dimensional domains,
the $C^k$-property can be reformulated
in the expected way:
\begin{la}\label{fin-dim-calc}
Let $E$ be a finite-dimensional
real vector space, $U\sub E$
be a locally convex subset with dense
interior, $F$ be a locally convex space,
and $k\in \N$.
Then the following conditions are equivalent
for a continuous map $f\colon U\to F$:
\begin{itemize}
\item[\rm(a)]
$f$ is $C^k$;
\item[\rm(b)]
$f$ is $C^1$
and $f'\colon U\to L(E,F)_b$, $x\mto f'(x)=df(x,\cdot)$
is $C^{k-1}$.
\end{itemize}
\end{la}
\begin{proof}
To see that (a) implies (b),
we may assume that $E\not=\{0\}$, excluding a trivial case.
Let $b_1,\ldots, b_n$ be a basis for~$E$.
Then
\[
\psi\colon L(E,F)_b\to F^n,\quad
\alpha\mto (\alpha(b_1),\ldots, \alpha(b_n))
\]
is an isomorphism of topological vector spaces.
If $f$ is $C^k$, then $f$ is $C^1$ and
\[
\psi\circ f'=(df(\cdot, b_1),\ldots,df(\cdot,b_n))
\]
is $C^{k-1}$ (by \cite[Lemma 1.4.6 and Proposition~1.4.10]{GaN}), entailing that
$f'=\psi^{-1}\circ (\psi\circ f')$
is $C^{k-1}$.

Conversely, assume that (b) holds.
The projections
$\pr_1\colon U\times E\to U$, $(x,y)\mto x$
and $\pr_2\colon U\times E\to E$, $(x,y)\mto y$
are smooth, being restrictions
of continuous linear maps $E\times E\to E$.
Since $E$ is finite dimensional and thus normable,
the evaluation map $\ve\colon L(E,F)_b\times E\to F$,
$(\alpha,y)\mto\alpha(y)$ is continuous
and thus smooth, being bilinear.
Hence $df=\ve\circ (f'\circ \pr_1,\pr_2)$
is $C^{k-1}$, using \cite[Lemma~1.4.6 and Proposition 1.4.10]{GaN}.
\end{proof}
If $f\colon X\to Y$
is a function between metric spaces $(X,d_X)$ and $(Y,d_Y)$,
we define
\[
\Lip(f):=\sup\left\{\frac{d_Y(f(x),f(y))}{d_X(x,y)}\colon \mbox{$x,y\in X$ with $x\not=y$}\right\}\in [0,\infty].
\]
Thus $f$ is Lipschitz continuous if and only if
$\Lip(f)<\infty$, and $\Lip(f)$
is the smallest Lipschitz constant for~$f$ in this
case.
We shall use a version of the Lipschitz Inverse
Function Theorem.
\begin{la}\label{lip-inv}
Let $(E,\|\cdot\|)$
be a normed space, $M\sub E$ a subset
and $f\colon M\to E$ a Lipschitz
continuous mapping
such that $\Lip(f)<1$.
Then $\phi:=\id_M+f\colon \! M\to E$\linebreak
is injective
and $\phi^{-1}\colon \phi(M)\to M$
is Lipschitz continuous with
\begin{equation}\label{Lipinv}
\Lip(\phi^{-1})\;\leq \; \frac{1}{1-\Lip(f)}.
\end{equation}
If, moreover, $(E,\|\cdot\|)$
is a Banach space and~$M$ is open in~$E$,
then $\phi(M)$ is open in~$E$.
\end{la}
\begin{proof}
For all $x,y\in M$,
\[
\|\phi(x)-\phi(y)\|\geq \|x-y\|-\|f(x)-f(y)\|
\geq (1-\Lip(f))\|x-y\|.
\]
Thus $\phi(x)=\phi(y)$ entails $x=y$,
and also (\ref{Lipinv}) follows.
If $(E,\|\cdot\|)$
is a Banach space and $M$
is open in~$E$, then $\phi$
is an open map, as a consequence of \cite[Theorem~5.3]{IMP};
notably, $\phi(M)$ is open.
\end{proof}
\begin{numba}
Let $E_1$, $E_2$, and $F$ be locally convex spaces,
$U_1\sub E_1$ as well as $U_2\sub E_2$
be locally convex subsets with dense interior,
and $r,s\in \N_0\cup\{\infty\}$.
Following~\cite{AaS},
we say that a map $f\colon U_1\times U_2\to F$
is $C^{r,s}$
if $f$ is continuous and, for all $i,j\in \N_0$
with $i\leq r$ and $j\leq s$, there exists a continuous map
\[
d^{(i,j)}f\colon U_1\times U_2\times (E_1)^i\times (E_2)^j\to F
\]
such that the iterated directional derivatives
\[
(D_{(v_i,0)}\cdots D_{(v_1,0)}D_{(0,w_j)}\cdots D_{(0,w_1)}f)(x,y)
\]
exist for all $x\in U_1^0$, $y\in U_2^0$,
$v_1,\ldots, v_i\in E_1$, $w_1,\ldots,w_j\in E_2$,
and coincide with $d^{(i,j)}f(x,y,v_1,\ldots, v_i,w_1,\ldots, w_j)$,
using the interiors $U_1^0$ and $U_2^0$ (see also \cite{GaN}).
\end{numba}
We mention that $C^{r,s,t}$-maps $f\colon U_1\times U_2\times U_3\to F$
can be defined analogously~\cite{Alz}.
\section{Proof of Proposition~\ref{prop-2}}
By definition of the compact-open $C^k$-topology,
the map
\[
C^k_{\str}(M,E)\to C(M\times E,E), \quad f\mto df
\]
is continuous if we endow $C(M\times E,E)$
with the compact-open topology. Hence
\[
U_k=\{f\in C^k_{\str}(M,E)\colon df(M\times \wb{B}^E_1(0))\sub B^E_1(0)\}
\]
is open in $C^k_{\str}(M,E)$. Moreover, $0\in U_k$.
We prove the remaining assertion
of Proposition~\ref{prop-2}
by induction on the dimension of~$E$.
The case $\dim(E)=0$
is trivial, since $f=0$
for all $f\in C^k_{\str}(M,E)$
in this case and thus $\id_M+f=\id_M\in\Diff^k_{\face}(M)$.
Let $\dim(E)>0$ now and assume
that the assertion has been
established for all
finite-dimensional
vector spaces of dimension $<\dim(E)$,
in place of~$E$.
Given $f\in U_k$, our goal
is to show that $\phi:=\id_M+f$
is a face-respecting $C^k$-diffeomorphism
of~$M$.
As a first step,
we show that $\phi(F)=F$
for each face $F\not=M$
of~$M$. Pick an element $x$
in the algebraic interior $\algint(F)$
of~$F$
(called the ``relative interior''
in \cite{Bro,Web}).
Then $F=M(x)$; we endow
\[
E(x)=\Spann_\R(F-x)
\]
with the norm induced by $\|\cdot\|$.
For each $z\in\algint(F)$,
we have $f(z)\in E(z)=\aff(F)-z=\aff(F)-x=E(x)$.
Then $f(z)\in E(x)$ for all $z\in F$,
as $f$ is continuous and $\algint(F)$ is dense in~$F$.
Since $F\not=M$, we have $\dim(E(x))<\dim(E)$.
Now $N:=F-x$ is a polytope
in $E(x)$, with non-empty interior relative~$E(x)$.
Moreover, the map
\[
g\colon N\to E(x),\quad y\mto f(x+y)
\]
is $C^k$ (using \cite[Lemma~1.4.16]{GaN})
and Lipschitz with $\Lip(g)\leq 1$.
If $y\in N$,
let $N(y)$ be the smallest face
of~$N$ such that $y\in N(y)$.
Then $x+N(y)$ is the smallest face of~$F$
which contains $x+y$, and hence
also the smallest face $M(x+y)$
of~$M$ which contains $x+y$.
As a consequence,
\begin{eqnarray*}
g(y)=f(x+y)\in E(x+y)\! & \! = \! &\! \Spann_\R(M(x+y)-(x+y))\\
\! &\! =\! &\! \Spann_\R(x+N(y)-(x+y))=\Spann_\R(N(y)-y)
\end{eqnarray*}
holds, showing that $g\in C^k_{\str}(N,E(x))$.
By the inductive hypothesis,
we have $(\id_N+g)(N)=N$.
Since
$(\id_M+f)(z)=z+f(z)=x+(z-x)+f(x+(z-x))=x+(\id_N+g)(z-x)$
for all $z\in F$,
we deduce that
\[
(\id_M+f)(F)=x+(\id_N+g)(F-x)=x+(\id_N+g)(N)=x+N=F,
\]
concluding the first step of the proof.
As the second step, we prove $\phi(M)\sub M$.
If this was false,
we could find $x\in M$
such that $\phi(x)\not\in M$.
Now~$M$ being an intersection of half-spaces,
we would find a linear functional $\lambda\colon E\to\R$
and $a\in \R$ such that
\[
\lambda(M)\sub \;]{-\infty},a]\quad
\mbox{and}\quad \lambda(\phi(x))>a.
\]
Then $\lambda\not=0$,
entailing that $\lambda$ is an open map.
Since~$M$ is compact and $\lambda\circ \phi$ is continuous,
after replacing~$x$ if necessary we may assume that
\[
\lambda(\phi(x))=\max \lambda(\phi(M)).
\]
Since $\phi(F)=F\sub M$
for each proper face~$F$
of~$M$ by the first step, we must
have $x\in M^0$.
Since $\phi(M^0)$ is open in~$E$
by Lemma~\ref{lip-inv},
$\lambda(\phi(M^0))$
is an open neighbourhood of $\lambda(\phi(x))$
in~$\R$, contradicting the maximailty
of $\lambda(\phi(x))$. Hence~$x$ cannot exist
and $\phi(M)\sub M$ must hold.\\[2.3mm]
As the third step,
we show that $\phi(M)=M$.
Since $\partial M$
is the union of all proper faces~$F$
of~$M$ and $\phi(F)=F$ by Step~1,
we have $\phi(\partial M)=\partial M$.
Now $M$ being closed in~$E$,
we have $M=M^0\cup\partial M$
with $M^0\cap \partial M=\emptyset$.
We already observed that~$\phi(M^0)$ is open in~$E$,
by Lemma~\ref{lip-inv};
thus $\phi(M^0)\sub M^0$.
Since $\phi(M)$ is compact and hence closed
in~$M$, the intersection
\[
\phi(M)\cap M^0=(\phi(M^0)\cup \phi(\partial M))\cap M^0=\phi(M^0)
\]
is closed in~$M^0$.
But $\phi(M^0)$ is also open in $E$,
and hence open in~$M^0$.
Since $M^0$ is connected
and $\phi(M^0)\not=\emptyset$,
we deduce that $M^0=\phi(M^0)$.
As a consequence, $\phi(M)=\phi(M^0\cup\partial M)=M^0\cup\partial M=M$.\\[2.3mm]
By Lemma~\ref{lip-inv}, $\phi\colon M\to M$
is a homeomorphism. By Steps~1 and~3, $\phi(F)=F$
for each face~$F$ of~$M$,
whence~$\phi$ is face respecting.\\[2.3mm]
The inversion map $j\colon \GL(E)\to\GL(E)$, $\alpha\mto \alpha^{-1}$
is smooth on the general linear group
$\GL(E):=L(E,E)^\times$.
For each $x\in M^0$,
we have $\phi'(x)-\id_E=f'(x)$
with $\|f'(x)\|_{\op}<1$,
whence $\phi'(x)\colon E\to E$
is invertible (by means of Neumann's series).
Thus $\phi|_{M^0}$ is a local $C^k$-diffeomorphism at~$x$,
by the Inverse Function Theorem.
As a consequence, the bijection $\phi|_{M^0}\colon M^0\to M^0$
is a $C^k$-diffeomorphism.
Now $\phi^{-1}\colon M\to E$ is a continuous map
and $\phi^{-1}|_{M^0}=(\phi|_{M^0})^{-1}$ is $C^1$
with
\[
(\phi^{-1}|_{M^0})'=((\phi|_{M^0})^{-1})'=j\circ (\phi|_{M_0})'\circ (\phi|_{M^0})^{-1}.
\]
By the preceding, the map
\[
g:=j\circ \phi'\circ \phi^{-1}\colon M\to L(E,E)_b
\]
is a continuous extension of $(\phi^{-1}|_{M^0})'\colon M^0\to L(E,E)_b$.
Then $g^\wedge\colon U\times E\to E$,
$(x,y)\mto g(x)(y)=\ve(g(x),y)$ is a continuous extension of
$d(\phi^{-1}|_{M^0})$,
using that the evaluation map
$\ve\colon
L(E,E)_b\times E\to E$
is continuous.
Hence $\phi^{-1}$ is $C^1$ with $d(\phi^{-1})=g$ and
thus
\begin{equation}\label{make-indu}
(\phi^{-1})'=j\circ \phi'\circ \phi^{-1}.
\end{equation}
By induction on $\ell\in \N$
with $\ell\leq k$, we may assume that
$\phi^{-1}$ is $C^{\ell-1}$.
Then $(\phi^{-1})'$ is $C^{\ell-1}$, by (\ref{make-indu}),
whence $\phi^{-1}$ is $C^\ell$, by Lemma~\ref{fin-dim-calc}.
Thus $\phi^{-1}$ is $C^k$,
entailing that $\phi \in \Diff^k_{\face}(M)$
and hence $f=\phi-\id_M\in\Omega_k$.
Thus $U_k\sub \Omega_k$.
\section{Proof of Proposition~\ref{main-prop}}
We first show that $\Omega_k$ is open.
\begin{la}\label{omega-open}
\hspace*{-.3mm}For each $k\in\N\cup\{\infty\}$,
the set
$\Omega_k:=\{\phi-\id_M\colon \phi\in \Diff^k_{\face}(M)\}$
is an open $0$-neighbourhood in~$C^k_{\str}(M,E)$.
\end{la}
\begin{proof}
If $\phi\in\Diff^k_{\face}(M)$,
then $\phi-\id_M\colon M\to E$
is a $C^k$-map.
For each $x\in M$, we have
$\phi(x)-x\in\phi(M(x))-x=M(x)-x\sub E(x)$,
showing that $\phi-\id_M$ is a stratified
vector field and thus $\phi-\id_M\in C^k_{\str}(M,E)$.
By Proposition~\ref{prop-2}, $\Omega_k$
is a $0$-neighbourhood in $C^k_{\str}(M,E)$.
It remains to show that $\Omega_k$ is open.
Let $V:=C^k(M,E)$.
For each $g\in \Omega_k$, the map
\[
R_g\colon C^k(M,E)\to C^k(M,E),\quad f\mto f\circ (\id_M+ g)
\]
is continuous linear, by \cite[Proposition~1.7.11]{GaN}.
Now $h:=(\id_M + g)^{-1}-\id_M\in
\Omega_k$.
Since $\id_M+h=(\id_M+g)^{-1}$,
we see that $R_h\circ R_g=R_g\circ R_h=\id_V$.
Thus~$R_g$ is an isomorphism
of topological vector spaces,
with $(R_g)^{-1}=R_h$.
Also
\[
\tau\colon C^k(M,E)\to C^k(M,E),\quad f\mto \id_M+f
\]
is a homeomorphism. We deduce that
\[
r_g:=\tau^{-1}\circ R_g\circ \tau \colon C^k(M,E)\to C^k(M,E),\quad
f\mto g+f\circ (\id_M+g)
\]
is a homeomorphism with $r_g^{-1}=
\tau^{-1}\circ R_h\circ \tau=r_h$.\vspace{-.3mm}
Using that $\phi:=
\id_M+g\in \Diff^k_{\face}(M)$
is face respecting,
we now show that
\begin{equation}\label{compo-is-str}
r_g(f)\in C^k_{\str}(M,E)\quad\mbox{for each $\, f\in C^k_{\str}(M,E)$.}
\end{equation}
To this end, let $f\in C^k_{\str}(M,E)$.
For $x\in M$,
the image $\phi(M(x))=M(x)$
is a face containing the element $\phi(x)$,
whence $M(\phi(x))\sub M(x)$.
Replacing~$x$ with~$\phi(x)$ and~$\phi$ with its inverse,
the same argument shows that
$M(x)=M(\phi^{-1}(\phi(x)))\sub M(\phi(x))$.
Thus $\phi(M(x))=M(\phi(x))$.
As a consequence, $\aff M(x)=\aff M(\phi(x))$ and hence
\[
E(x)=(\aff M(x))-x=(\aff M(\phi(x)))-\phi(x)=E(\phi(x)).
\]
Now $r_g(f)(x)=g(x)+f(x+g(x))=g(x)+f(\phi(x))\in E(x)+E(\phi(x))=E(x)$,
establishing~(\ref{compo-is-str}).
By~(\ref{compo-is-str}), the map $r_g$
restricts to a self-map
\[
\rho_g\colon C^k_{\str}(M,E)\to C^k_{\str}(M,E)
\]
which is continuous as we endowed $W:=C^k_{\str}(M,E)$
with the topology
induced by $C^k(M,E)$.
As $\rho_g\circ\rho_h=\rho_h\circ \rho_g=\id_W$,
we deduce that $\rho_g$
is a homeomorphism
with $(\rho_g)^{-1}=\rho_h$.
We claim that
\[
\rho_g(\Omega_k)\sub \Omega_k.
\]
If this is true, then $\Omega_k$
is a $g$-neighbourhood in $C^k_{\str}(M,E)$
(which completes
the proof),
as $\Omega_k$ is a $0$-neighbourhood,
$\rho_g$ a homeomorphism,
and $\rho_g(0)=g$.\linebreak
To establish the claim,
let $f\in \Omega_k$.
Then $\id_M+f\in \Diff^k_{\face}(M)$
and hence $(\id_M+f)\circ (\id_M+g)\in \Diff^k_{\face}(M)$,
the latter being closed under composition.
Thus
$\rho_g(f)=g+f\circ (\id_M+g)=(\id_M+f)\circ (\id_M+g)-\id_M\in \Omega_k$.
\end{proof}
We write $\Diff^r(K):=\{\phi\in C^r(K,K)\colon(\exists \psi\in C^r(K,K))\;
\phi\circ\psi=\psi\circ\phi=\id_K\}$
in the next lemma.
\begin{la}\label{GN-thmC}
Let $F$ be a locally convex space,
$U\sub F$ be a locally convex subset with dense
interior, $E$ be a finite-dimensional
real vector space,
$K\sub E$ be a compact convex subset
with non-empty interior and $r\in\N_0\cup\{\infty\}$.
If a map $f\colon U\times K\to K\sub E$
is~$C^r$ and $f_x:=f(x,\cdot)\in \Diff^r(K)$
for all $x\in U$, then also the following map is~$C^r$:
\[
g\colon U\times K\to K,\quad (x,z)\mto (f_x)^{-1}(z).
\]
\end{la}
\begin{proof}
If $r=0$, then the graph of~$f$ is closed
in $U\times K\times K$.
As a consequence, the graph of~$g$
is closed in $U\times K\times K$
(being obtained from the former by flipping the second and third component).
Since~$K$ is compact, continuity of~$g$ follows (see, e.g.,
\cite[Lemma~2.1]{GN2}). For $r\in \N\cup\{\infty\}$,
we can repeat the proof of
\cite[Theorem~C]{GN2}
without changes.
\end{proof}
{\bf Proof of Proposition~\ref{main-prop}.}
The evaluation map
\[
\ve_k\colon C^k(M,E)\times M\to E,\quad (f,x)\mto f(x)
\]
is $C^{\infty,k}$
and thus $C^{\ell,k}$,
by \cite[Proposition~3.20]{AaS}.
Likewise, the evaluation
map $\ve_{k+\ell}\colon C^{k+\ell}(M,E)\times M\to E$
is $C^{\infty,k+\ell}$.
Let us show that the map
\[
H_{k,\ell}\colon C^{k+\ell}(M,E)\times \Omega_k
\to C^k(M,E),\quad
(f,g)\mto g+f\circ (\id_M+g)
\]
is $C^\ell$, for all
$k\in\N\cup\{\infty\}$ and $\ell\in \N_0\cup\{\infty\}$.
It suffices to show that
\[
\Gamma_{k,\ell}\colon C^{k+\ell}(M,E)\times \Omega_k\to C^k(M,E),\quad
(f,g)\mto f\circ (\id_M+g)
\]
is $C^\ell$, since
$H_{k,\ell}(f,g)=g+\Gamma_{k,\ell}(f,g)$.
This will hold if we can show that
$\Gamma_{k,\ell}$ is $C^{\infty,\ell}$
(see \cite[Lemma~3.15]{AaS}).
By \cite[Theorem~3.20]{Alz}, it suffices to show that
\[
\Gamma_{k,\ell}^\wedge\colon C^{k+\ell}(M,E)\times \Omega_k\times M\to E,
\quad
(f,g,x)\mto f(x+g(x))=
\ve_{k+\ell}(f,x+\ve_k(g,x))
\]
is $C^{\infty,\ell,k}$.
Now $\Omega_k\times M\to M$, $(g,x)\mto x$
is a smooth map and hence $C^{\ell,k}$.
As also~$\ve_k$ is $C^{\ell,k}$,
we find that
\[
h_2\colon \Omega_k\times M\to M,\quad (g,x)\mto x+g(x)=x+\ve_k(g,x)
\]
is $C^{\ell,k}$. The identity map $h_1\colon C^{k+\ell}(M,E)\to C^{k+\ell}(M,E)$,
$f\mto f$ is $C^\infty$. Since $\ve_{k+\ell}$ is $C^{\infty,k+\ell}$,
the Chain Rule in the form \cite[Lemma~3.16]{Alz}
% dort Druckfehler leider in zweitletzer Zeile!
shows that
\[
\Gamma_{k,\ell}^\wedge=\ve_{k+\ell}\circ (h_1\times h_2)
\]
is $C^{\infty,\ell,k}$.
Thus $H_{k,\ell}$ is $C^\ell$,
whence also $H_{k,\ell}|_{\Omega_{k+\ell}\times\Omega_k}$
is~$C^\ell$.
Now
\[
H_{k,\ell}(f,g)=(\id_M+f)\circ (\id_M+g)-\id_M\in \Omega_k\sub
C^k_{\str}(M,E)
\]
for all $f\in \Omega_{k+\ell}\sub\Omega_k$
and $g\in\Omega_k$, since $\Diff_{\face}^k(M)$
is closed under composition.
Hence $H_{k,\ell}|_{\Omega_{k+\ell}\times\Omega_k}$
co-restricts to a map
\[
h_{k,\ell}\colon \Omega_{k+\ell}\times \Omega_k\to\Omega_k\sub C^k_{\str}(M,E)
\]
which is $C^\ell$ as the vector subspace $C^k_{\str}(M,E)$
of $C^k(M,E)$ is closed (see \cite[Lemma~1.3.19]{GaN}).
Then also $c_{k,\ell}=\Phi_k^{-1}\circ h_{k,\ell}\circ (\Phi_{k+\ell}\times\Phi_k)$
is~$C^\ell$.\\[2.3mm]
For the discussion of the inversion maps,
note that, for each $k\in\N\cup\{\infty\}$,
$\Diff^k_{\face}(M)$
is a smooth submanifold
of $C^k(M,E)$ modeled
on the closed vector
subspace $C^k_{\str}(M,E)$ of $C^k(M,E)$.
In fact, since $\Omega_k$ is an open subset
of $C^k_{\str}(M,E)$
whose topology is induced by the compact-open
$C^k$-topology on $C^k(M,E)$,
we find an open subset $Q\sub C^k(M,E)$
with $Q\cap C^k_{\str}(M,E)=\Omega_k$.
Then $P:=\id_M+Q=\{\id_M+f\colon f\in Q\}$
is an open subset of $C^k(M,E)$
and
\[
\kappa\colon P\to Q,\quad \phi\mto\phi-\id_M
\]
is a chart for $C^k(M,E)$ such that
$\kappa(P\cap \Diff^k_{\face}(M))=\Omega_k=Q\cap C^k_{\str}(M,E)$.
As a consequence,
for
%all
$k\in \N\cup\{\infty\}$
and $\ell\in\N_0\cup\{\infty\}$,
the map $\iota_{k,\ell}$
will be~$C^\ell$ to $\Diff^k_{\face}(M)$
if we can show that~$\iota_{k,\ell}$ is
$C^\ell$ as a map to $C^k(M,E)$
(see \cite[Lemma~3.2.7]{GaN}).
By \cite[Theorem~3.25]{Alz}, the latter will hold if
\[
\iota_{k,\ell}^\wedge\colon \Diff^{k+\ell}_{\face}(M)\times M\to E,\quad
(\phi,z)\mto \iota_{k,\ell}(\phi)(z)=\phi^{-1}(z)
\]
is $C^{\ell,k}$. But $\iota_{k,\ell}^\wedge$
even is $C^{k+\ell}$;
in fact,
\[
f\colon \Diff^{k+\ell}_{\face}(M)\times M\to E,\quad
(\phi,x)\mto \phi(x)
\]
is $C^{\infty,k+\ell}$ (and hence $C^{k+\ell}$)
as the restriction to a submanifold
of the evaluation\linebreak
mapping
$\ve_{k+\ell}\colon C^{k+\ell}(M,E)\times M\to E$,
which is $C^{\infty,k+\ell}$.
Since
$\Diff^{k+\ell}_{\face}(M)$
is $C^\infty$-diffeomorphic to the open set
$\Omega_{k+\ell}\sub C^{k+\ell}_{\str}(M,E)$,
the $C^{k+\ell}$-property of
$\iota_{k,\ell}^\wedge\colon (\phi,z)\mto f(\phi,\cdot)^{-1}(z)$
follows from Lemma~\ref{GN-thmC}.\\[2.3mm]
Note that $c_{\infty,\infty}$
is the group multiplication of $\Diff^\infty_{\face}(M)$
and $\iota_{\infty,\infty}$ the group inversion.
As both of these mappings are smooth, $\Diff^\infty_{\face}(M)$
is a Lie group. $\,\square$
\section{Proof of Theorem~\ref{main-thm}}
Two lemmas will be useful for the proof of Theorem~\ref{main-thm}.
\begin{la}\label{der-to-face}
Let $E$ be a finite-dimensional
real vector space and $M\sub E$ be a convex polytope
with non-empty interior.
If $r>0$ and $\gamma\colon [0,r[\;\to M$
is a $C^1$-map,
then there exists $\ve>0$
such that
\[
\gamma(0)+t\gamma'(0)\in M\quad\mbox{for all $\,t\in [0,\ve]$.}
\]
\end{la}
\begin{proof}
We may assume
%that
$E\not=\{0\}$,
omitting a trivial case.
There are a finite set $\Lambda\not=\emptyset$
of non-zero linear functionals
$\lambda \colon E \to \R$
and numbers $a_\lambda\in\R$ such that
\begin{equation}\label{intersect-halfspace}
M=\{x\in E\colon (\forall \lambda\in \Lambda)\;\;
\lambda(x)\leq a_\lambda\}.
\end{equation}
Let $\Lambda_0:=\{\lambda\in\Lambda\colon \lambda(\gamma(0))=a_\lambda\}$.
Then
\[
\lambda(\gamma'(0))=(\lambda\circ\gamma)'(0)=
\lim_{t\to 0_+}\frac{\lambda(\gamma(t))-\lambda(\gamma(0))}{t}
\leq 0
\]
for each $\lambda\in\Lambda_0$,
as $\lambda(\gamma(t))\leq a_\lambda$
and $\lambda(\gamma(0))=a_\lambda$.
As a consequence,
\[
\lambda(\gamma(0)+t\gamma'(0))=a_\lambda+t\lambda(\gamma'(0))\leq a_\lambda
\]
for all $\lambda\in \Lambda_0$ and $t\geq 0$.
For all $\lambda\in \Lambda\setminus\Lambda_0$,
we have $\lambda(\gamma(0))<a_\lambda$.
As $\Lambda\setminus\Lambda_0$
is a finite set, we find $\ve>0$
such that $\lambda(\gamma(0))+t\lambda(\gamma'(0))\leq a_\lambda$
for all $t\in [0,\ve]$ and all $\lambda\in\Lambda\setminus\Lambda_0$.
For all $t\in [0,\ve]$,
we then have $\lambda(\gamma(0)+t\gamma'(0))\leq a_\lambda$
for all $\lambda\in\Lambda$.
Thus $\gamma(0)+t\gamma'(0)\in M$, by~(\ref{intersect-halfspace}).
\end{proof}
\begin{la}\label{face-into-face}
Let $E$ and $F$ be finite-dimensional
real vector spaces,
$M\sub E$ and $N\sub F$ be convex polytopes
with non-empty interior
and $\phi\colon U\to V$
be a $C^1$-map between
open subsets $U\sub M$ and $V\sub N$.
For $x\in M$,
let $M(x)$ be the face of~$M$
generated by~$x$,
and $E(x):=\Spann_\R(M(x)-x)$.
For $y\in N$, let $N(y)$ be the face of~$N$
generated by~$y$, and $F(y):=\Spann_\R(N(y)-y)$.
For $x\in U$, we~have:
\begin{itemize}
\item[\rm(a)]
$\phi'(x)(E(x))\sub F(\phi(x))$;
\item[\rm(b)]
If
$\phi'(x)\colon E\to F$ is injective,
then $\dim E(x)\leq \dim F(\phi(x))$;
\item[\rm(c)]
If $\phi\colon U\to V$
is a $C^1$-diffeomorphism,
then $\phi'(x)(E(x))=F(\phi(x))$.
\end{itemize}
\end{la}
\begin{proof}
(a) Let $w\in E(x)$.
Since $x\in \algint M(x)$,
there exists $r>0$ such that $x\pm tw\in M(x)\sub M$
for all $t\in [0,r]$.
After shrinking~$r$,
we may assume that $x\pm tw\in U$
for all $t\in [0,r]$ and thus $\phi(x\pm tw)\in N$.
Since $\pm \phi'(x)(w)=\frac{d}{dt}\big|_{t=0}\phi(x\pm t w)$,
Lemma~\ref{der-to-face}
provides $\ve>0$ such that
\[
\phi(x)\pm t\phi'(x)(w)\in N\quad\mbox{for all $\, t\in [0,\ve]$.}
\]
Notably, $v_+:=\phi(x)+\ve \phi'(x)(w)\in N$
and $v_-:=\phi(x)-\ve \phi'(x)(w)\in N$.
Since $(1/2)v_++(1/2)v_-=\phi(x)\in N(\phi(x))$
and $N(\phi(x))$ is a face of~$N$,
we deduce that $v_+,v_-\in N(\phi(x))$.
Thus $\ve \phi'(x)(w)=v_+ -\phi(x) \in N(\phi(x))-\phi(x)\sub F(\phi(x))$,
whence also $\phi'(x)(w)\in F(\phi(x))$.\\[2.3mm]
(b) is immediate from (a).\\[2.3mm]
(c) Set $y:=\phi(x)$.
By (a), we have $\phi'(x)(E(x))\sub F(y)$
and $(\phi^{-1})'(y)(F(y))\sub E(\phi^{-1}(y))=E(x)$.
Since $(\phi^{-1})'(y)=(\phi'(x))^{-1}$, applying $\phi'(x)$
to the latter inclusion we get
$F(y)\sub \phi'(x)(E(x))$, whence
$F(y)=\phi'(x)(E(x))$.
\end{proof}
Lemma~\ref{face-into-face}\,(c)
is analogous to the well-known boundary
invariance for manifolds
with corners (as in \cite[Theorem~1.2.12\,(a)]{MRO}).
\begin{defn}\label{defn-index}
Let $E$ be a finite-dimensional real
vector space of dimension~$n$
and $M\sub E$ be a convex polytope
with non-empty interior.
The \emph{index} of $x\in M$
is defined as
\[
\ind_M(x):=\dim (E/E(x))=n-\dim E(x)=n-\dim M(x),
\]
where $M(x)$ is the face of~$M$
generated by~$x$ and $E(x):=\Spann_\R (M(x)-x)$.
For $i\in \{0,1,\ldots, n\}$, we define
\[
\partial_i(M):=\{x\in M\colon \ind_M(x)=i\}.
\]
\end{defn}
In the situation of Definition~\ref{defn-index},
the following holds.
\begin{la}\label{partial-boundary}
For each $i\in \{0,1,\ldots,n\}$,
the set $\partial_i(M)$ is an
$(n-i)$-dimensional smooth submanifold of~$E$.
If $\cF_{n-i}(M)$ is the finite set
of all faces of~$M$
of dimension $n-i$, then
the topology induced by~$M$ on~$\partial_i(M)$
makes the latter the topological sum
of the algebraic interiors $\algint(F)$
for $F\in \cF_{n-i}(M)$.
The connected components
of $\partial_i(M)$
are the sets $\algint(F)$
for $F\in\cF_{n-i}(M)$;
they are open and closed in $\partial_i(M)$.
\end{la}
\begin{proof}
If $F,G\in\cF_{n-i}(M)$ 
such that $F\not=G$ and $F\cap G\not=\emptyset$,
then $F\cap G$ is a face of~$M$
of dimension $<n-i$
and $F\cap G\sub F\setminus\algint(F)$
as well as $F\cap G\sub G\setminus \algint(G)$,
entailing that $\algint(F)\cap \algint(G)=\emptyset$.
Thus
\begin{equation}\label{disj-union}
\algint(F)\cap\algint(G)=\emptyset\quad\mbox{for
all $\,F,G\in\cF_{n-i}(M)$ such that $\,F\not=G$.}
\end{equation}
Let $F\in\cF_{n-i}(M)$.
For each $G\in\cF_{n-i}(M)$,
we have $\algint(F)\cap G\sub F\cap G\sub F\setminus \algint(F)$
by the preceding and thus $\algint(F)\cap G=\emptyset$.
Thus
\[
K:=\bigcup_{G\in\cF_{n-i}(M)\setminus\{F\}}G
\]
is a compact subset of~$E$ such that $\algint(F)\cap K=\emptyset$.
Note that
$F=H\cap M$ for some hyperplane~$H$ in~$E$
(faces of polyhedra being exposed), whence
\[
F=M\cap \aff(F).
\]
%a fortiori.
Now $\algint(F)=U\cap \aff(F)$
for some open subset $U\sub E$.
After replacing $U$
by its intersection with the open
set $E\setminus K$ which contains $\algint(F)$,
we may assume that
$U\cap \algint(G)=\emptyset$
for all $G\in\cF_{n-i}(M)\setminus \{F\}$
and hence
\[
\partial_i(M)\cap U=U\cap \algint(F)=\algint(F),
\]
showing that $\algint(F)$ is open
in $\partial_i(M)$.
The topology induced by~$E$
therefore makes $\partial_i(M)$
the topological sum of the $\algint(F)$ with $F\in \cF_{n-i}(M)$.
If we choose $x\in \algint(F)$,
then $E(x)=\aff(F)-x$
and
\[
\phi\colon U\to U-x,\quad y\mto y-x
\]
is a $C^\infty$-diffeomorphism between
open subsets of~$E$ such that
\[
\phi(U\cap \partial_i(M))=\phi(\algint(F))=\phi(U\cap \aff(F))
=(U-x)\cap E(x).
\]
Thus $\partial_i(M)$ is a submanifold of~$E$
of dimension $\dim(E(x))=n-i$.
\end{proof}
\begin{la}\label{fin-index}
Let $E$ be a finite-dimensional
real vector space, $M\sub E$
be a convex polytope with non-empty interior,
and $k\in\N\cup\{\infty\}$.
Then $\Diff^k_{\face}(M)$
is a normal subgroup of finite index
in $\Diff^k(M)$.
\end{la}
\begin{proof}
We write $\cF(M)$
for the set of all faces of~$M$.
Let $n:=\dim(E)$ and $\phi\in \Diff^k(M)$.
By Lemma~\ref{face-into-face}\,(c),
we have $\ind_M(\phi(x))=\ind_M(x)$
for each $x\in M$, whence
$\phi(\partial_i(M))\sub \partial_i(M)$
for each $i\in \{0,1,\ldots,n\}$
and thus $\phi(\partial_i(M))=\partial_i(M)$,
as $M$ is the disjoint union of
$\partial_0(M),\ldots,\partial_n(M)$
and $\phi$ is surjective.
Since $\partial_i(M)$
is a submanifold of~$E$,
the inclusion map $\partial_i(M)\to E$
is smooth, and thus also
the inclusion map $j_i\colon \partial_iM\to M$.
Thus $\phi|_{\partial_i(M)}=\phi\circ j_i$
is~$C^k$.
As this map has image in $\partial_i(M)$,
which is a submanifold of~$E$,
we deduce that
\[
\phi_i:=\phi|_{\partial_i(M)}\colon \partial_i(M)\to\partial_i(M)
\]
is~$C^k$. Note that $(\phi^{-1})_i$
is inverse to~$\phi_i$, whence $\phi_i$
is a $C^k$-diffeomorphism
and hence a homeomorphism.
For each $F\in \cF_{n-i}(M)$,
the algebraic interior $\algint(F)$
is open and closed in the topological
sum $\partial_i(M)$.
As a consequence, $\phi(\algint(F))=\phi_i(\algint(F))$
is both open and closed in $\partial_i(M)$.
Being also non-empty and connected,
$\phi(\algint(F))$ is a connected component
of $\partial_i(M)$
and thus $\phi(\algint(F))=\algint(G)$
for some $G\in \cF_{n-i}(M)$.
As a consequence,
\[
\phi(F)=\phi(\overline{\algint(F)})=\overline{\phi(\algint(F))}
=\overline{\algint(G)}=G.
\]
Thus $\phi_*(F):=\phi(F)\in\cF(M)$
for each $F\in \cF(M)$.
Since $(\id_M)_*$
is the identity map on $\cF(M)$
and $(\phi\circ\psi)_*=\phi_*\circ\psi_*$
for all $\phi,\psi\in \Diff^k(M)$, we get a group homomorphism
\[
\pi\colon \Diff^k(M)\to \, \Sym(\cF(M)),\quad \phi\mto\phi_*
\]
to the group of all permutations of the finite set $\cF(M)$.
Thus $\Diff_{\face}^k(M)=\ker(\pi)$
is a normal subgroup of $\Diff^k(M)$ of finite index.
\end{proof}
{\bf Proof of Theorem~\ref{main-thm}.}
In view of the local description of Lie group
structures (analogous to
Proposition~18 in \cite[Chapter III, \S1, no.\,9]{Bou}),
we need only show that the group homomorphism
\[
I_\psi\colon \Diff^\infty_{\face}(M)\to\Diff^\infty_{\face}(M),\quad
\phi\mto \psi\circ \phi\circ \psi^{-1}
\]
is smooth for each $\psi\in \Diff(M)$.
Since $\Diff^\infty_{\face}(M)$ is a smooth submanifold
of $C^\infty(M,E)$,
we need only show that $I_\psi$ is smooth as a map to $C^\infty(M,E)$,
which will hold if we can show that the map
\[
I_\psi^\wedge\colon \Diff^\infty_{\face}(M)\times M\to E,\quad
(\phi,x)\mto I_\psi(\phi)(x)=\psi(\phi(\psi^{-1}(x)))
\]
is smooth (see \cite[Theorem~3.25 and Lemma~3.22]{AaS}).
We know that the evaluation map
$\ve\colon C^\infty(M,E)\times M\to E$,
$(f,x)\mto f(x)$ is smooth
(see \cite[Proposition 3.20 and Lemma~3.22]{AaS}).
Since $\Diff^\infty_{\face}(M)$ is a submanifold of $C^\infty(M,E)$,
also the restriction of $\ve$ to a map
\[
\Diff^\infty_{\face}(M)\times M\to E
\]
is smooth. As this map takes its values in~$M$,
also its corestriction
\[
\theta\colon \Diff^\infty_{\face}(M)\times M\to M,\quad
(\phi,x)\mto \phi(x)
\]
is smooth. The formula
\[
I_\psi^\wedge(\phi,x)=\psi(\theta(\phi,\psi^{-1}(x)))
\]
now shows that $I_\psi^\wedge$ is a smooth function
of $(\phi,x)$, as required. $\,\square$
\begin{rem}\label{boundary-fix}
If $E=\R^n$ and $M\sub E$ is a polytope
with non-empty interior,~then
\[
\Diff_{\partial M}(M):=\{\phi\in \Diff(M)\colon \mbox{$\phi(x)=x$
for all $x\in \partial M$}\}
\]
is a subgroup of~$\Diff^\infty_{\face}(M)$.
Since
\[
C^\infty_{\partial M}(M,E):=\{f\in C^\infty(M,E)\colon
\mbox{$f(x)=0$ for all $x\in \partial M$}\}
\]
is a closed vector subspace of $C^\infty_{\str}(M,E)$
and the chart
\[
\Phi_\infty\colon \Diff^\infty_{\face}(M)\to \Omega_\infty\sub C^\infty_{\str}(M,E)
\]
satisfies $\Phi_\infty(\Diff_{\partial M}(M))=\Omega_\infty\cap
C^\infty_{\partial M}(M,E)$,
we see that $\Diff_{\partial M}(M)$
is a submanifold (and hence a Lie subgroup)
of $\Diff^\infty_{\face}(M)$ modeled on $C^\infty_{\partial M}(M,E)$,
and thus also of $\Diff(M)$
(in the sense of~\cite[Definition 3.1.10]{GaN}).
\end{rem}
\begin{rem}
Let $E$ be a finite-dimensional real vector space,
$n:=\dim(E)$,
$M\sub E$ be a convex polytope with non-empty interior,
and $k\in\N_0\cup\{\infty\}$.
Using Lemma~\ref{partial-boundary},
we obtain the following more intuitive characterization:
A function $f\in C^k(M,E)$
is stratified (and thus in $C^k_{\str}(M,E)$)
if and only if $(\id_M,f)(\partial_i(M))\sub T(\partial_i(M))$
for each $i\in\{0,\ldots, n\}$,
identifying $T(\partial_i(M))$ with a subset
of $TM=M\times\R^n$ as usual.
That is, $f|_{\partial_i(M)}$
can be considered as a $C^k$-vector field
of the $(n-i)$-dimensional $C^\infty$-manifold $\partial_i(M)$
for all $i\in\{0,\ldots,n\}$.
\end{rem}
\begin{rem}\label{polyhedral-mfd}
Given $n\in\N_0$,
we define a \emph{locally polyhedral $C^\infty$-manifold
of dimension~$n$} as a Hausdorff
topological space~$M$,
together with a set $\cA$
of homeomorphisms $\phi\colon U_\phi\to V_\phi$
from open subsets $U_\phi\sub M$
onto open subsets of a polytope $P_\phi\sub\R^n$
with non-empty interior,
such that $\cA$ is a \emph{polyhedral smooth atlas}
in the sense that
$\bigcup_{\phi\in \cA}U_\phi=M$
and
$\phi\circ \psi^{-1}$ is $C^\infty$
for all $\phi,\psi\in\cA$,
and~$\cA$
is maximal among such atlasses.
Notably, $\cA$ is a rough $C^\infty$-atlas
in the sense of~\cite{GaN},
whence~$M$ can be considered
as a smooth manifold with rough boundary
in the sense of~\cite{GaN}.
By Lemma~\ref{face-into-face}, for all $x\in M$
we have
\[
\ind_{P_\phi}(\phi(x))=\ind_{P_\psi}(\psi(x))
\]
for all $\phi,\psi\in \cA$
such that $x\in U_\phi$ and $x\in U_\psi$,
whence
\[
\ind_M(x):=\ind_{P_\phi}(\phi(x))
\]
is a well-defined integer
in $\{0,1,\ldots, n\}$.
If we set
\[
\partial_i(M):=\{x\in M\colon \ind_M(x)=i\},
\]
then $\phi(U_\phi\cap \partial_i(M))=V_\phi\cap \partial_i(P_\phi)$
is an $(n-i)$-dimensional
submanifold of~$\R^n$ for each $\phi\in \cA$,
entailing that $\partial_i(M)$ is an $(n-i)$-dimensional
manifold (and a submanifold of~$M$
in the sense of \cite[Definition 3.5.14]{GaN};
cf.\ also \cite[Remark 3.5.16\,(a)]{GaN}).
\end{rem}
\appendix
\section{Non-existence of Michor-type local
additions}
If $M$ is a finite-dimensional
smooth manifold with corners
of dimension~$n$,
one can define its tangent
bundle $TM$ with fibre $\R^n$
and obtains the subset ${}^iTM$
of all tangent vectors $v$
of the form
\[
v=\dot{\gamma}(0)\in T_{\gamma(0)}M
\]
for some smooth curve $\gamma\colon [0,1[\;\to M$,
the so-called \emph{inner} tangent vectors
(see \cite[p.\,20]{Mic}).
If $M=[0,\infty[^k\times \R^{n-k}$
with $k\in \{0,\ldots,n\}$,
we identify $TM$ with $M\times\R^n$
as usual;
then
\[
{}^iTM=\{(x,y)\in M\times\R^n\colon (\forall j\in \{1,\ldots,k\})\;
x_j=0\;\Rightarrow\; y_j\geq 0\},
\]
writing $x=(x_1,\ldots, x_n)$ and $y=(y_1,\ldots, y_n)$
in components. Hence ${}^iTM$ is a convex subset
of $\R^n\times \R^n$
and its interior $M^0\times\R^n$ is dense
in ${}^iTM$.
Notably, ${}^iTM$ is a locally convex subset
of $\R^{2n}$ with dense
interior.
Since every $n$-dimensional
smooth manifold~$M$ with corners locally looks like
$[0,\infty[^n$,
we see that ${}^iTM$
locally looks like ${}^iT([0,\infty[^n)$,
whence ${}^iTM$ is a smooth manifold
with rough boundary in the sense of~\cite[Chapter~3]{GaN}.
It therefore makes sense
to speak about smooth functions
on ${}^iTM$.
(A smaller class of
modeling sets
and corresponding generalized manifolds
is proposed in \cite[p.\,20, Remark]{Mic}).
According to \cite[\S10.2, p.\,90]{Mic},
local additions on a smooth manifold~$M$
with corners are defined as follows:
\begin{numba}\label{local-add-michor}
A local addition $\tau$ on $M$ is a smooth mapping 
$\tau\colon {}^iTM \to M$ satisfying
\begin{itemize}
\item[(A1)]
$(\tau,\pi_{TM})\colon  {}^iTM \to M\times M$
is a diffeomorphism onto an open 
neighbourhood of the diagonal in $M\times M$;
\item[(A2)]
$\tau(0_x)=x$ for all $x\in M$.
\end{itemize}
\end{numba}
Here $0_x\in T_xM$ is the $0$-vector in
the tangent space $T_xM$ for $x\in M$.
Michor claims that every
smooth manifold with corners
admits a local addition
(see \cite[p.\,90, Lemma]{Mic}).
However:
\begin{prop}
$M:=[0,\infty[$
does not admit a local
addition in the sense
defined by Michor,
as in {\rm\ref{local-add-michor}}.
\end{prop}
\begin{proof}
We identify $TM$ with $M\times\R$.
Thus
\[
{}^iTM=(\{0\}\times [0,\infty[)\cup (]0,\infty[\,\times\R).
\]
Suppose that $\tau\colon {}^iTM\to M$
was a local addition. We shall derive a contradiction.
Since $\phi\colon {}^iTM\to M\times M$, $(x,y)\mto (\tau(x,y),x)$
is a diffeomorphism onto on open neighbourhood of
the diagonal,
the Jacobi matrix $J_\phi(0,0)$
must be invertible.
By~(A2), we have $\tau(x,0)=x$,
whence $\frac{\partial \tau}{\partial x}(0,0)=1$.
Thus
\[
J_\phi(0,0)=\left(
\begin{array}{cc}
1 & \frac{\partial \tau}{\partial y}(0,0)\\
1 & 0
\end{array}\right),
\]
and invertibility implies that $\frac{\partial \tau}{\partial y}(0,0)\not=0$.
Since $\tau(0,y)\in [0,\infty[$
for all $y\in [0,\infty[$,
we have $\tau(0,y)\geq 0$ and
thus $\frac{\partial \tau}{\partial y}(0,0)\geq 0$,
using that $\tau(0,0)=0$.
Thus $\frac{\partial \tau}{\partial y}(0,0)>0$.
Choose $\theta>0$
such that
\[
\theta\, \frac{\partial \tau}{\partial y}(0,0)>1;
\]
thus $\ve:=\theta\frac{\partial\tau}{\partial y}(0,0)-1>0$.
Consider the smooth map
\[
h\colon [0,\infty[\,\to M\sub \R,\quad
t\mto \tau(t,-\theta t).
\]
Then
\[
h'(0)=\frac{\partial \tau}{\partial x}(0,0)-\theta\frac{\partial \tau}{\partial y}(0,0)
=- \ve.
\]
Since $h(0)=0$, Taylor's
Theorem yields
\[
h(t)=-\ve t +R(t)
\]
with $R(t)/t\to 0$ for $t\to 0$.
There exists $\delta> 0$ such that
$|R(t)|/t<\ve$ for all $t\in \;]0,\delta]$,
whence $R(t)<\ve t$ and thus
\[
h(t)<-\ve t+\ve t=0.
\]
This contradicts $h(t)\geq 0$,
which holds as $h(t)\in M$.
\end{proof}
It would not help to assume that, instead,
the local addition $\tau$
is only defined on an open neighbourhood
$\Omega$ of the $0$-section in ${}^iTM$.
In fact,
such a neighbourhood~$\Omega$ of
$[0,\infty[\,\times \{0\}$
in ${}^iT([0,\infty[)$
contains $]0,\rho[\,\times\,]{-\rho},\rho[$
for some $\rho>0$.
There exists $\delta\in \; ]0,\rho]$
such that $\theta \delta <\rho$
and thus $(t,-\theta t)\in\Omega$
for all $t\in [0,\delta[$.
Hence $h(t):=\tau(t,-\theta t)\in M$.
After shrinking~$\delta$
if necessary,
a contradiction is obtained
as in the preceding proof.\\[2.3mm]
An analogous remark applies in the following
more general situation.
We\linebreak
consider smooth manifolds
with corners as in \cite{Mic} here,
which are assumed finite-dimensional
and paracompact.
\begin{prop}
If $M$ is a
smooth manifold with corners such that $\partial M\not=\emptyset$,
then~$M$
does not admit a smooth local addition in Michor's
sense.
\end{prop}
\begin{proof}
Suppose we could find
a smooth local addition $\tau\colon {}^iTM\to M$
in Michor's sense. We derive a contradiction.
Let $n\in\N$ be the dimension of~$M$.
For $p\in \partial M$, there exist
$k\in \{1,\ldots, n\}$,
an open $p$-neighbourhood
$U\sub M$ and
a $C^\infty$-diffeomorphism
$\phi\colon U\to [0,\infty[^k\times\R^{n-k}=:V$
such that $\phi(p)=0$.
Let
\[
F:={}^iTV\sub TV
=V\times \R^n
\]
be the set of all $(x,y)\in V\times\R^n$
with $x=(x_1,\ldots, x_n)$ and $y=(y_1,\ldots, y_n)$
such that $y_j\geq 0$ for all
$j\in \{1,\ldots, k\}$ such that $x_j=0$.
Then~$F$
is a convex subset of $\R^n\times\R^n$ with dense interior.
Moreover, $\tau^{-1}(U)$ is an open
$0_p$-neighbourhood in ${}^iTM$ and
\[
Q:=T\phi((TU)\cap \tau^{-1}(U))
\]
is an open $(0,0)$-neighbourhood in $F\sub V\times \R^n$,
whence $Q$ is a locally convex subset
of $\R^n\times\R^n$ with dense interior.
We obtain a smooth map
\[
\sigma := \phi\circ \tau\circ T\phi^{-1}|_Q\colon Q \to V\sub \R^n
\]
such that
\[
\psi:=(\sigma,\pr_1)\colon Q \to V\times V,\quad (x,y)\mto (\sigma(x,y),x)
\]
is a $C^\infty$-diffeomorphism onto an open neighbourhood
of the diagonal in $V\times V$.
For some $\rho>0$,
we have $P:=\;]{-\rho},\rho[^{2n}\cap F\sub Q$.
Write $\sigma|_P=(\sigma_1,\ldots,\sigma_n)$
in terms of the components
$\sigma_1,\ldots,\sigma_n\colon P\to\R$.
Using the partial maps $\sigma_0:=\sigma(0,\cdot)\colon[0,\infty[^k\times\R^{n-k}\to\R^n$
and $\sigma^0:=\sigma(\cdot,0)\colon [0,\rho[^k\times ]{-\rho},\rho[^{n-k}
\to\R^n$
of $\sigma|_P$,
the Jacobian of $\psi$ at $(0,0)\in P\sub \R^n\times\R^n$
can be regarded as the block matrix
\[
J_\psi(0,0)=\left(
\begin{array}{cc}
J_{\sigma^0}(0) & J_{\sigma_0}(0)\\
{\bf 1} & {\bf 0}
\end{array}\right).
\]
The invertibility of the Jacobian
implies that the final $n$ columns
must be linearly independent.
Notably, we must have
\[
\frac{\partial \sigma_1}{\partial y_j}(0,0)\not=0
\]
for some $j\in \{1,\ldots, n\}$.
Let $e_1,\ldots,e_n\sub\R^n$ be the standard basis vectors.
For small $t\geq 0$,
we have $(0,te_j)\in Q$, whence
$\sigma(0,te_j)\in V$, entailing that
$\sigma_1(0,te_j)\geq 0$ and thus
\[
\frac{\partial \sigma_1}{\partial y_j}(0,0)=
\lim_{t\to 0_+}\frac{\sigma_1(0,te_j)}{t}\geq 0,
\]
exploiting that $\sigma_1(0,0)=0$.
There exists $\theta>0$
such that
\[
\ve:=\theta\, \frac{\partial \sigma_1}{\partial y_j}(0,0)
-1 >0.
\]
There exists $\delta>0$ such that
\[
t (1,\ldots,1,{-\theta e_j}) \in Q
\]
for all $t\in [0,\delta]$
(with $1$ in the first $n$
slots).
Then
\[
h\colon [0,\delta]\to\R,\quad t\mto \sigma_1(t(1,\ldots, 1,-\theta e_j))
\]
is a smooth function such that
\[h'(0)=\sum_{i=1}^n\underbrace{\frac{\partial \sigma_1}{\partial x_i}(0,0)}_{=\delta_{1,i}}-\theta\,\frac{\partial\sigma_1}{\partial y_j}(0,0)=
1-\theta\frac{\partial\sigma_1}{\partial y_j}(0,0)=-\ve<0,
\]
where $\delta_{1,i}$ denotes Kronecker's delta.
Since $h(0)=0$,
after shrinking~$\delta$, we can achieve
that $h(t)<0$ for all $t\in \;]0,\delta]$.
Thus $\sigma_1(\delta(1,\ldots,1,-\theta e_j))=h(\delta)<0$,
contradicting the fact $\sigma(Q)\sub V$
and thus $\sigma_1(x,y)\geq 0$
for all $(x,y)\in Q$.
\end{proof}
\noindent{\bf Helge Gl\"{o}ckner}, Universit\"{a}t Paderborn,
Institut f\"{u}r Mathematik,\\
Warburger Str.\ 100, 33098 Paderborn, Germany;
\,{\tt glockner@math.upb.de}
\end{document}